\documentclass[english]{article}

\usepackage{babel}

\usepackage[a4paper,top=2cm,bottom=2cm,left=3cm,right=3cm]{geometry}

\usepackage[textsize=footnotesize]{todonotes}
\usepackage{tikz-cd} 
\usepackage{amssymb,amsfonts,amsmath,amsthm}
\usepackage[colorlinks=false, allcolors=blue]{hyperref}

\usepackage{newfloat}
\DeclareFloatingEnvironment[fileext=lod,name=Diagram,placement=p]{fdiagram}

\usepackage{multirow}

\usepackage[boxed, noline, noend, linesnumbered]{algorithm2e}
\DontPrintSemicolon
\usepackage[T1]{fontenc}
\usepackage[scaled=0.85]{beramono}

\newcommand{\algkw}[1]{{\ttfamily\bfseries #1}}

\SetNlSty{algkw}{}{}

\SetKwInOut{KwIn}{\algkw{Input}}
\SetKwInOut{KwOut}{\algkw{Output}}
\SetKw{Ret}{\algkw{return}}
\SetKwProg{fn}{\algkw{function}}{}{}

\SetKwIF{If}{ElseIf}{Else}%
  {\algkw{if}}{\algkw{then}}{\algkw{else if}}{\algkw{else}}{\algkw{end if}}

\SetKwFor{For}{\algkw{for}}{\algkw{do}}{\algkw{end for}}
\SetKwFor{While}{\algkw{while}}{\algkw{do}}{\algkw{end while}}

\numberwithin{table}{section}
\numberwithin{fdiagram}{section}
\numberwithin{equation}{section}

\newtheorem{theorem}{Theorem}[section]

\newtheorem{corollary}[theorem]{Corollary}
\newtheorem{lemma}[theorem]{Lemma}
\newtheorem{proposition}[theorem]{Proposition}
\theoremstyle{definition}
\newtheorem{definition}[theorem]{Definition}

\newtheorem*{convention*}{Convention}
\newtheorem{example}[theorem]{Example}
\theoremstyle{remark}
\newtheorem{remark}[theorem]{Remark}
\newtheorem*{remark*}{Remark}
\newtheorem{problem}[theorem]{Problem}

\newcommand{\Z}{{\mathbb Z}}
\newcommand{\Q}{{\mathbb Q}}
\newcommand{\R}{{\mathbb R}}
\newcommand{\N}{{\mathbb N}}

\DeclareMathOperator{\Aut}{Aut}

\DeclareMathOperator{\Hol}{Hol}
\DeclareMathOperator{\Aff}{Aff}
\DeclareMathOperator{\GL}{GL}

\DeclareMathOperator{\Isom}{Isom}
\DeclareMathOperator{\Soc}{Soc}
\DeclareMathOperator{\diag}{diag}

\usepackage[noblocks]{authblk}

\usepackage[nobysame,bibtex-style, initials]{amsrefs}
\BibSpec{misc}{%
  +{}{\PrintAuthors} {author}
  +{.}{ \textit}     {title}
  +{,}{ }            {date}
  +{.}{ }            {url}
  +{,}{}             {note}
}

\begin{document}
\title{Cofinite integral braces and crystallographic groups}
\author{Rafa{\l} Lutowski}
\author{Andrzej Szczepa\'nski}
\affil{Institute of Mathematics, Faculty of Mathematics, Physics and Informatics, University of Gda\'nsk\par Wita Stwosza 57, 80-308 Gda\'nsk, Poland}
\author{Anna Zamojska-Dzienio}
\affil{Faculty of Mathematics and Information Science, Warsaw University of Technology, Koszykowa 75, 00-662 Warsaw, Poland}
\maketitle

\renewcommand{\thefootnote}{\fnsymbol{footnote}} 
\footnotetext{$^*$ Corresponding author: \url{rafal.lutowski@ug.edu.pl}}
\footnotetext{2020 \emph{Mathematics Subject Classification.} Primary: 17B38, Secondary: 20H15, 81R05, 57N45.}
\footnotetext{\emph{Keywords and phrases.} Cofinite integral brace, crystallographic group, Yang-Baxter equation, structure group.}
\renewcommand{\thefootnote}{\arabic{footnote}}

\begin{abstract}
A left brace is cofinite integral (CIB) if its additive group is free abelian of finite rank and its quotient by the socle is finite. Rump introduced such braces and showed that the multiplicative group of every CIB is a crystallographic group. We study CIBs from the point of view of crystallographic groups. We prove an analogue of the Bieberbach theorems: in each dimension there are only finitely many isomorphism classes of CIBs. For a fixed crystallographic group, we show that its CIB structures correspond to the orbits of a finite group action on a set of bijective 1‑cocycles. We use this description to develop algorithms that enumerate them. Finally, we give a complete census of CIBs on crystallographic and Bieberbach groups up to dimension 6.   
\end{abstract}
\section{Introduction}
Let $X$ be a non-empty (finite) set. A function $r:X\times X\to X\times X$ is a \emph{set-theoretical solution to the quantum Yang-Baxter equation} if

\begin{equation}\label{Eq:SetBraidEqn}
(r\times \text{id}_X)(\text{id}_X\times r)(r\times \text{id}_X)=(\text{id}_X\times r)(r\times \text{id}_X)(\text{id}_X\times r).
    \end{equation}
The term \emph{solution} refers to a set-map pair $(X, r)$ such that $r$ satisfies \eqref{Eq:SetBraidEqn}. For the solution $(X,r)$ we write
\[
r(x,y) = (\sigma_x(y), \tau_y(x)).
\] If for each $x,y\in X$, $\sigma_x,\tau_y\colon X\to X$ are bijective, then we say $(X, r)$ is  \emph{nondegenerate}. We call a solution $(X, r)$ \emph{involutive} if $r^2=\text{id}_{X\times X}$. 

In what follows we adopt the following convention:
\begin{convention*} All solutions we study in this paper are set-theoretical, non-degenerate and involutive so we will call them simply \emph{solutions}. Moreover, set $X$ in a solution $(X,r)$ is always finite.
\end{convention*}
Braces are ring-like structures, generalizing Jacobson radical rings, introduced by Rump \cite{Ru07} to study solutions of the Yang-Baxter equation. Here, we use the equivalent definition of a brace given in \cite{CJO14}. Namely, a (left) \emph{brace} is an algebraic structure $(A,+,\circ)$ consisting of an abelian group $(A,+)$ and a group $(A,\circ)$ connected by the compatibility condition 
\[a\circ (b+c)+a=a\circ b+a\circ c\]
that holds for all $a,b,c\in A$. The \emph{socle} of a left brace $A$ is 
\[\Soc(A)=\{a\in A\mid a+b=a\circ b \text{ for all }b\in A\}.\]
One can consider a quotient left brace $A/\Soc(A)$ on residue classes $\{\Soc(A)+a\colon a\in A\}$ with operations induced in a standard way from $(A,+,\circ)$ . 

In \cite{Ru24}, Rump introduces and investigates \emph{cofinite integral} braces (CIB), i.e., left braces $A$ such that $A/\Soc(A)$ is finite and the additive group of $A$ is free abelian of finite rank. In particular, he shows that the multiplicative group of a CIB $A$ is always crystallographic (see \cite{Ru24}*{Proposition 4}) and describes cofinite braces whose additive group is free abelian of rank one and two (\cite{Ru24}*{Section 5}). In dimension one, there are just two integral braces. Both are cofinite, with the infinite cyclic or the infinite dihedral group as multiplicative group.
In dimension two, there are 17 crystallographic groups (sometimes called wallpaper groups)
\cite{Sz25}. All are multiplicative groups of braces \cite{Ru24}*{Corollary of Theorem 3}, but only 12 of them can be obtained from CIB \cite{Ru24}*{Theorem 3}. Furthermore, all ten torsion-free 3-dimensional crystallographic groups are multiplicative groups of CIBs \cite{Ru24}*{Theorem 4}.

In this paper we study general relations between CIB and crystallographic groups of dimension $n\geq 3.$ The paper is organized as follows. 

In Section \ref{S:basic}, we review the necessary background on braces, solutions to the Yang--Baxter equation, and crystallographic groups. We also study two-sided cofinite integral braces and prove that any such brace is necessarily trivial, i.e., $A=\Soc(A)$ (Theorem~\ref{T:two-sided}). 

In Section \ref{S:isomorphism}, we analyze affine representations and bijective $1$-cocycles associated with cofinite integral braces. In particular, we establish a criterion showing that an $n$-dimensional crystallographic group $\Gamma$ defines a CIB if and only if its translational projection $\varphi\colon \Gamma \to \R^n$ is injective and $\varphi(\Gamma)$ is a subgroup of $\mathbb{R}^n$ (Proposition~\ref{P:criterion}). Using this framework, we verify that every generalized Hantzsche-Wendt group admits a cofinite integral brace structure (Lemma~\ref{lemma:hantzsche-wendt-existence}) and prove an analogue of the Bieberbach theorem showing that there are only finitely many isomorphism classes of CIBs in any given dimension (Theorem~\ref{T:Bieberbach}). 

In Section \ref{S:classification}, we provide a cohomological classification of cofinite integral braces with a fixed crystallographic multiplicative group $\Gamma$, establishing that their isomorphism classes correspond bijectively to the orbits of the action of the stabilizer $N_\gamma$ on the set of bijective $1$-cocycles whose images form groups (Theorem~\ref{theorem:classification}). 

In Section \ref{section:computer-calculations}, we develop computational algorithms to classify CIB structures on crystallographic groups. We establish criteria to bound the maximal exponent of the additive group of the quotient brace (Lemma~\ref{lemma:socle-restriction}, Algorithm~\ref{alg:maximal-exponents}), formulate a search procedure over the lattice of integer-valued coboundaries (Algorithm~\ref{alg:finding-cibs-with-exponent}), and present an effective $p$-subgroup reduction method for groups of composite order (Lemma~\ref{L:reduction}). 

In Section \ref{S:GHW}, we focus on generalized Hantzsche-Wendt groups. We show first that for every CIB structure on any such group $\Gamma$, the exponent of the additive group of the quotient brace $\Gamma/\mathbb{Z}^n$ is at most~4 (Proposition~\ref{proposition:exponent-in-hw-groups}). This is followed by the explicit classification demonstrating that there are exactly three isomorphism classes of cofinite integral braces on the three-dimensional Hantzsche-Wendt group (Theorem~\ref{T:3dHW}). 

In Section \ref{S:low_dim}, we present computational classification results for crystallographic and Bieberbach groups up to dimension~6 obtained via computer algebra system GAP \cite{GAP4} (Tables~\ref{table:cibs-on-crystallographic-groups} and~\ref{table:cibs-on-bieberbach-groups}). We confront these results with previous low-dimensional observations in \cite{Ru24}, pointing out new phenomena such as the existence of asymmorphic Bieberbach groups without cofinite integral brace structures in dimension~4 and higher (Remark~\ref{asymmorphic}). 

In Section \ref{S:conclusions}, we present concluding remarks and propose open problems.

\section{Basic definitions and properties}\label{S:basic}
\subsection{Braces}
For the convenience of the reader, we review in this section the necessary background on braces and crystallographic groups.
\begin{definition}[\cite{CJO14}]
A \emph{left brace} is a set $A$ with two operations $+$ and $\circ$ such that $(A, +)$ is an abelian group, $(A,\circ)$ is a group and
\begin{equation}\label{eq:leftBr}
a\circ( b + c ) + a = a\circ b + a\circ c ,
\end{equation}
for all $a,b,c\in A$. We call $(A, +)$ the \emph{additive group} and $(A,\circ)$ the \emph{multiplicative group} of the left brace.

A \emph{right brace} is deﬁned similarly, replacing condition \eqref{eq:leftBr} by
\begin{equation}\label{eq:rightBr}
( b + c )\circ a + a = b\circ a + c\circ a,
\end{equation}
for all $a,b,c\in A$.
\end{definition}

There is a bijective correspondence between left braces and right braces (see e.g. \cite{Cedo}*{Definition 2.2}). It is easy to check that in a left (or right) brace $A$, the identity element $1$ of the multiplicative
group $(A,\circ)$ is equal to the neutral element $0$ of the additive group $(A,+)$. A \emph{two-sided brace} is a left brace that is also a right brace. Given any abelian group $(A,+)$, it is easy to check that $(A,+,+)$ is a brace. Such a brace is called a \emph{trivial brace}.

Let $A$ and $B$ be left braces. A map $f \colon A \to B$ is a \emph{brace morphism} if it is a group homomorphism with respect to both operations, that is, $f(a+b) = f(a)+f(b)$ and $f(a \circ b) = f(a) \circ f(b)$ for all $a, b \in A$. If $f$ is bijective, it is called a \emph{brace isomorphism}.

An additive subgroup $I$ of a brace $A$ is said to be a {\em right ideal} 
if for every $a\in I$ and $b\in A$,
\[
a \circ b - b \in I.
\]
In particular, every right ideal $I$ is a subbrace of $A$ (that is, $a, b \in I$ implies $a+b \in I$ and $a \circ b \in I$). Moreover, if also $b \circ a - b \in I$ for all $a \in I$ and $b \in A$, then $I$ is called an \emph{ideal}. Equivalently, an ideal $I$ of $A$ is a right ideal such that $(I,\circ)$ is a normal subgroup of $(A,\circ)$.

The kernel $\ker f := f^{-1}(0)$ of a brace morphism $f:A\to B$ is an ideal,
and the image $f(A)$ is a subbrace of $B.$  Conversely, any ideal $I$ of a brace $A$ gives rise
to a surjective brace morphism $A\to A/I$ with kernel $I.$
As in the case of rings, the quotient brace $A/I$ consists of the residue classes $I + a = I\circ a,$ and
its operations are induced from those of $A.$

\begin{definition}[\cite{CJO14}]
The \emph{socle} of a left brace $A$ is the ideal defined by
\[\Soc(A)=\{a\in A\mid a+b=a\circ b \text{ for all }b\in A\}.\]
\end{definition}
\noindent
We shall also need the following definition.

\begin{definition}[\cite{Sz25}]
A \emph{crystallographic group} of dimension $n$ is a discrete subgroup of $\Isom(\R^n)\cong \mathrm{O}(n) \ltimes \mathbb{R}^n$ such that the orbit space $\R^n/\Gamma$ is compact. If $\Gamma$ is torsion-free, it is called a \emph{Bieberbach group}.
\end{definition}

\subsection{Solutions of the Yang-Baxter equation}

\begin{remark}
Recall that for the solution $(X,r)$ we write
\[
r(x,y) = (\sigma_x(y), \tau_y(x)),
\]
where both mappings $\sigma_x, \tau_y$ are bijections for all $x,y\in X$. Involutivity of a solution implies that it is enough to consider only the mappings $\sigma_x$, since one can use them and their inverses to express the mappings $\tau_x$ (\cite{Ru07}*{}).\end{remark}

\begin{definition}
The \emph{structure group} of a solution $(X,r)$ is a group $G(X,r)$ with the presentation
\[
G(X,r) := \langle X \mid xy = uv \Leftrightarrow r(x,y) = (u,v) \rangle.
\]
\end{definition}

\begin{theorem}[\cite{GIVB98}*{Theorem 1.6}]
If $(X,r)$ is a finite %
solution, then $G(X,r)$ is a Bieberbach group.
\end{theorem}

\begin{theorem}[\cite{ESS99}*{Proposition 2.3}, \cite{ALV20}*{Theorem 2.1}]
Let $(X,r)$ with $|X|=n \in \mathbb{N}$ be a solution. %
Then there exists a monomorphism $G(X,r) \to O(n,\Z) \ltimes \Z^n \subset \GL(n+1,\Z)$, given by the formula
\[
x \mapsto (\sigma_x, t_x) \mapsto
\begin{bmatrix}
\sigma_x & t_x\\
0 & 1
\end{bmatrix}.
\]
\end{theorem}
\begin{remark}
\label{remark:notation}
Note that we identify $X$ with a basis of the free abelian group $\Z^n$. For every $x \in X$, $t_x$ is the translation by the basis vector $x$, and the permutation $\sigma_x$ is identified with its matrix in the basis $X$ of $\Z^n$.    
\end{remark}

\begin{lemma}
Let $(A,+,\circ)$ be a left brace, and let a map $\lambda \colon (A,\circ) \to \Aut(A,+)$ be given by $a \mapsto \lambda_a$, where
\[
\lambda_a(b) = -a + a \circ b,
\]
for every $a,b \in A$. Then $\lambda$ is a group homomorphism.
\end{lemma}

\begin{lemma}
Let $(A,+,\circ)$ be a left brace. Then
\[
\Soc(A) = \ker \lambda.
\]
\end{lemma}

By \cite{CV26}*{Theorem 8.3.1} we have the following result.
\begin{corollary}
If $(G,\circ)$ is a group and $A$ is a left $G$-module, then a bijective $1$-cocycle $\varphi \colon G \to A$ defines a left brace structure on $G$, with the addition given by
\[
g + h := \varphi^{-1}(\varphi(g) + \varphi(h)).
\]
In other words, the left brace structure is unique with the property that $\varphi \colon (G,+) \to A$ is an isomorphism of abelian groups.
\end{corollary}

By \cite{CJO14}*{Theorem 1}, the structure group of the solution $(X,r)$ admits a unique brace structure, which, in a very natural way, agrees with the solution. By the above corollary, we describe this structure in terms of the bijective $1$-cocycle.

\begin{corollary}
The left brace structure on the group $G(X,r) = \langle (\sigma_x, t_x) : x \in X \rangle$, is given by the bijective cocycle $\varphi \colon G(X,r) \to A$, where $A$ is the free abelian group generated by $X$. The cocycle is defined by the formula
\[
\varphi(\sigma_x, t_x) = x
\]
for $x \in X$.
\end{corollary}

By \cite{ALV20}*{Theorem 2.2}, the multiplicative group of the quotient brace $G(X,r) / \Soc(G(X,r))$ is the holonomy group of the Bieberbach group $G(X,r)$, hence it is finite. In \cite{Ru24} Rump introduces the following definition.

\begin{definition}[\cite{Ru24}*{Definition 1}]
Let $(A,+,\circ)$ be a left brace. It is called
\begin{itemize}
    \item \emph{integral} of dimension $n$, if $(A,+)$ is a free abelian group of rank $n$;
    \item \emph{cofinite}, if $A/\Soc(A)$ is a finite brace.
\end{itemize}
A brace satisfying both conditions is called a \emph{cofinite integral brace} (CIB).
\end{definition}

By the remarks preceding the above definition, we have a first natural example.

\begin{example}\label{E:structure_group}
Every structure group of a finite, non-degenerate, involutive solution gives rise to a cofinite integral brace.
\end{example}

The Rump's definition may be thought of as an extension of the investigation of left brace structures on the structure groups of the solutions. However, we retain control over the multiplicative group of cofinite integral braces.

\begin{proposition}[\cite{Ru24}*{Proposition 4}]
The multiplicative group of an $n$-dimensional cofinite integral brace $A$ is a crystallographic group of dimension $n$. The socle of $A$ is the maximal abelian normal subgroup of $(A,\circ)$.
\end{proposition}

By the Bieberbach theorems, every $n$-dimensional crystallographic group $\Gamma$ fits into the following short exact sequence
\begin{equation}
\label{equation:ses:crystallographic}
0 \longrightarrow L \longrightarrow \Gamma \longrightarrow G \longrightarrow 1,
\end{equation}
where the free abelian group 
$L$ of rank $n$ is the maximal abelian normal subgroup of $\Gamma$ and $G$ is a finite group. 

In the investigation of cofinite integral braces on low dimensional crystallographic groups, Rump comes to the following conclusions.

\begin{theorem}[\cite{Ru24}*{Theorem 3}]\label{T:wallpaper_CIB}
Let $\Gamma$ be a $2$-dimensional crystallographic group that fits into the sequence \eqref{equation:ses:crystallographic}. Then $\Gamma$ is a multiplicative group of a cofinite integral brace \emph{exactly} in one of the following cases:
\begin{enumerate}
    \item $\Gamma$ is \textbf{p4g}.
    \item $L$ has a non-trivial cyclic subgroup, which is normal in $\Gamma$.
    \item $G$ is cyclic of order $3$ or $4$. 
\end{enumerate}
\end{theorem}
\subsection{Two-sided cofinite integral braces}
\begin{theorem}\label{T:two-sided}
A cofinite integral left brace admits a compatible right brace structure if and only if it is a trivial brace.
\end{theorem}

\begin{proof}
Let $(A,+,\circ)$ be a cofinite integral left brace. The usual action $\lambda \colon (A,\circ) \to \Aut(A,+)$ can be used to define the $\circ$ action:
\[
\forall_{a,b \in A} a \circ b = a + \lambda_a(b).
\]
Now, the condition for a right brace structure, i.e.,
\[
(a+b) \circ c = a \circ c - c + b \circ c,
\]
translates to
\[
(\lambda_{a+b} + id)(c) = (\lambda_a + \lambda_b)(c)
\]
for all $a,b,c \in A$. This means that for every $a,b \in A$, we have the equality of maps
\[
\lambda_{a+b} + id = \lambda_a + \lambda_b
\]
and a simple induction argument shows that
\begin{equation}
\label{equation:multiplicity-of-a}
\lambda_{na} + (n-1)id = n \lambda_a
\end{equation}
for every $n \in \N$. Take $a \in A$. Since $A$ is cofinite, the quotient brace $A/\Soc(A)$ is finite, so $(A/\Soc(A),+)$ is a finite abelian group. Hence, there exists $n \in \N$ such that $na \in \Soc(A) = \ker \lambda$. Then $\lambda_{na}=id$, and equation \eqref{equation:multiplicity-of-a} gives us
\[
n \cdot id = n \cdot \lambda_a.
\]
Note that $\lambda_a$ is an automorphism of a free abelian group of finite rank, hence $\lambda_a = id$. Since $a \in A$ was arbitrary, we get $A = \Soc(A)$ and the proof is finished.
\end{proof}
\section{Isomorphism classes of cofinite integral braces}\label{S:isomorphism}

Recall that the holomorph $\Hol(G)$ of a group $G$ is the semidirect product $\Aut(G) \ltimes G$. It has been proved useful in the determination of braces with a given additive group, see for example \cite{Ba16}*{Proposition 2.3}. In particular, we have the following lemma.
\begin{lemma}
\label{lemma:natural-representation-of-brace}
Let $(A,+,\circ)$ be a left brace. Let $B(A) = \{ (\lambda_a,a) : a \in A \} \subset \Hol(A,+)$. The map
\[
\Lambda \colon A \to B(A),
\]
given by $a \mapsto (\lambda_a, a)$, is a brace isomorphism, where the brace structure on $B(A)$ is given by the bijective $1$-cocycle
\[
\varphi \colon B(A) \to (A,+), (\lambda_a,a) \mapsto a.
\]
\end{lemma}

\begin{remark}
\label{remark:bijective-cocycle}
By $\varphi \colon \Aff({\R}^n) = \GL(n,{\R}) \ltimes {\R}^n \to {\R}^n$ we will denote the projection to the second coordinate -- the translational part -- of the affine map:
\[
\varphi(X,x) := x.
\]
It is easy to see that $\varphi$ is a $1$-cocycle. In order to retain the clarity of notation we will also write $\varphi$ to denote its restriction to subgroups of $\Aff({\R}^n)$.
\end{remark}

\begin{corollary}
\label{corollary:representation-of-cib}
Let $(A,+,\circ)$ be a cofinite integral brace of dimension $n$. There exists a faithful representation
\[
\Lambda \colon (A,\circ) \to \GL(n,\Z) \ltimes \Q^n
\]
of the multiplicative group of $A$, such that:
\begin{enumerate}
    \item $\Lambda(A)$ is a cofinite integral brace with the corresponding 1-cocycle $\varphi$.
    \item $\Lambda(\Soc(A)) = \Z^n$.
    \item $\Lambda(A) \subset G \ltimes L$, where $\lambda(A) \cong G \subset \GL(n,\Z)$, $(A,+) \cong L \subset \Q^n$.
    \item $A/\Soc(A) \hookrightarrow G \ltimes L/\Z^n$.
    \item $L/\Z^n \cong (A/\Soc(A),+)$.
    \item $[L : \Z^n ] = |\lambda(A)| = |G|$.
\end{enumerate}
\end{corollary}

\begin{proposition}\label{P:criterion}
A pair $(\Gamma, \varphi)$, where $\Gamma \subset \Aff(\R^n)$ is a crystallographic group of dimension $n$, defines a cofinite integral brace if and only if $\varphi$ is one-to-one and the image $\varphi(\Gamma)$ is a group.
\end{proposition}

\begin{proof}
If $(\Gamma, \varphi)$ defines a cofinite integral brace, then the properties of $\varphi$ are fulfilled by definition. Now, if $\varphi$ is one-to-one and the image $\varphi(\Gamma)$ is a group, then the pair $(\Gamma, \varphi)$ defines a left brace. Hence, we have the action map $\lambda \colon (\Gamma, \circ) \to \Aut(\Gamma, +)$. For $a=(X,x), b=(Y,y) \in \Gamma$ we have
\[
\varphi(\lambda_a(b)) = -\varphi(a) + \varphi(a \circ b) = -\varphi(X,x) + \varphi(XY,Xy+x) = -x + Xy + x = Xy,
\]
hence
\[
\varphi(\Soc(\Gamma)) = \{ x \in {\R}^n : (I,x) \in \Gamma \}.
\]
We get that $(\Soc(\Gamma),\circ)$ is the maximal abelian normal subgroup of $\Gamma$ and by the Bieberbach theorems, $\Gamma$ is a cofinite brace. Moreover, $\varphi(\Soc(\Gamma))$ is free abelian group that is a subgroup of finite index in $\varphi(\Gamma)$, hence $\Gamma$ is integral of dimension $n$.
\end{proof}

\begin{remark}
For every $n \in \N$ we set
\[
\Aff_\Z(\Q^n) := \GL(n,\Z) \ltimes \Q^n.
\]
Note that by Bieberbach theorems \cite{Sz25}*{Theorem 2.4} not only a multiplicative group of CIB, but every crystallographic group of dimension $n$ is isomorphic to a subgroup of $\Aff_\Z(\Q^n)$. An isomorphism of any two such crystallographic groups can be realized by a conjugation in the group $\Aff(\Q^n)$. In addition, if one restricts to the crystallographic subgroups of $\Aff_\Z(\Q^n)$ for which the maximal abelian normal subgroup equals exactly $\Z^n$, then the isomorphism classes and conjugacy classes in $\Aff_\Z(\Q^n)$ coincide.
\end{remark}

\begin{corollary}
\label{corollary:cib-on-holonomy-group}
Let $\Gamma \subset \Aff_\Z(\Q^n)$ be a crystallographic group of dimension $n$. Assume that the maximal abelian normal subgroup of $\Gamma$ equals $\Z^n$. Let $G=\Gamma/\Z^n$ and $\overline{\varphi} \colon G \to \Q^n/\Z^n$ be given by
\[
\overline{\varphi} (X,x+\Z^n) := x+\Z^n,
\]
for $(X,x) \in \Gamma$.
Then the pair $(\Gamma, \varphi)$ defines a cofinite integral brace if and only if the pair $(\Gamma/\Z^n, \overline{\varphi})$ defines a brace, i.e., $\overline{\varphi}$ is one-to-one and $\overline{\varphi}(G)$ is a group.
\end{corollary}

\begin{proof}
We will show that properties of $\varphi$ and $\overline{\varphi}$ being one-to-one as well as the images of those maps being groups are equivalent.

Assume that $\overline{\varphi}$ is one-to-one. Let $(X,x),(Y,y) \in \Gamma$. We get
\[
\varphi(X,x) = \varphi(Y,y) \Rightarrow x = y \Rightarrow x+\Z^n = y+\Z^n \Rightarrow \overline{\varphi}(X,x+\Z^n) = \overline{\varphi}(Y,y+\Z^n).
\]
By assumption, we get $X=Y$ and $(X,x) = (Y,y)$. If $\varphi$ is one-to-one and $(X,x),(Y,y) \in \Gamma$ are such that
\[
\overline\varphi(X,x) = \overline\varphi(Y,y),
\]
then $x+\Z^n = y+\Z^n$, which obviously means that there exists $z \in \Z^n$ such that $x = y+z$. But then $\varphi(X,x) = \varphi(Y,y+z)$, hence $X=Y$ and $(X,x+\Z^n) = (Y,y+\Z^n)$.

If $\varphi(\Gamma)$ is a group, then $\overline\varphi(G) = \overline\varphi(\Gamma/\Z^n) = \varphi(\Gamma)/\Z^n$ is also a group. Suppose that $\overline\varphi(G)$ is a group and take $x,y \in \varphi(\Gamma)$. By the assumption we get that
\[
(x+\Z^n) - (y+\Z^n) = (x-y) + \Z^n \in \overline\varphi(G).
\]
Hence there exists $Z \in \GL(n,\Z)$ such that $(Z, x-y + \Z^n) \in \Gamma/\Z^n$, what exactly means that $(Z,x-y) \in \Gamma$ and finally $x-y \in \varphi(\Gamma)$.
\end{proof}

\begin{corollary}
\label{corollary:translational-part-image}
If a pair $(\Gamma,\varphi)$, as in the above corollary, defines a cofinite integral brace, then
\[
\varphi(\Gamma) \subset \frac{1}{|G|} \Z^n.
\]
\end{corollary}
\begin{proof}
$(\Gamma,\varphi)$ is CIB if and only if $G = (\Gamma/\Z^n,\overline{\varphi})$ is a brace and 
$\overline{\varphi}(G)\subset \Q^n/\Z^n$ is an abelian group of order $|G|$. Hence 
\[
\overline{\varphi}(G) \subset \frac{\frac{1}{|G|} \Z^n}{\Z^n},
\]
and $\varphi(G)\subset \frac{1}{\mid G\mid}\Z^n$.
\end{proof}

\begin{definition}
Let $\Gamma$ be a Bieberbach group of dimension $n$. $\Gamma$ is called a {\em generalized
Hantzsche-Wendt} group ($GHW$ for short) if and only if $\Gamma$ is a torsion free crystallographic group with holonomy $C_2^{n-1}$, an elementary abelian $2$-group of rank $n-1$.
\end{definition}

\begin{example}\label{example:promislow}
Promislow group is $GHW$ group. For more see \cite{Sz25}*{Introduction to Chapter 10}.
\end{example}

\begin{lemma}\label{lemma:hantzsche-wendt-existence}
Every $GHW$ group admits a cofinite integral brace structure.
\end{lemma}
\begin{proof}
Let $\Gamma$ be a $GHW$ group of dimension $n$. Denote by $\mathcal{D}$ the group of diagonal integer matrices in $\GL(n,\Z)$. Then by \cite{Sz25}*{Theorem 9.5} and \cite{MR02}*{Lemma 1.4}, up to isomorphism, $\Gamma \subset \mathcal{D} \ltimes \frac{1}{2}\Z^n$, and the maximal normal abelian subgroup of $\Gamma$ is $\Z^n$. The torsion-freeness of $\Gamma$ ensures that $\varphi$ is one-to-one. Let's consider the quotient group $\Gamma/\Z^n$, together with the one-to-one cocycle $\overline{\varphi}$. Because $\mathcal{D}$ acts trivially on $\frac{1}{2}\Z^n/\Z^n$, for $(X,x+\Z^n), (Y,y+\Z^n) \in \Gamma/\Z^n$ we have
\[
\overline\varphi(X,x+\Z^n)+ \overline\varphi(Y,y+\Z^n) = x+\Z^n + y + \Z^n = x+\Z^n + Xy + \Z^n = x+Xy + \Z^n = \overline\varphi(XY,x+Xy + \Z^n).
\]
Since every non-zero element of $\frac{1}{2}\Z^n/\Z^n$ is an involution, we get that 
$\overline{\varphi}(\Gamma/\Z^n)$ is a group, and hence our remark is proved.
\end{proof}

\begin{example}
Another example of an infinite family of Bieberbach groups admitting a cofinite integral braces structures are the ones with holonomy of prime order. Let $\Gamma$ be such a group, with the holonomy group $G$ of prime order $p$. Again, assume that $\Z^n$ is the maximal normal abelian subgroup of $\Gamma$. One can show that $\Gamma/\Z^n$ is generated by the block matrix of the form
\[
a := 
\begin{bmatrix}
    I_k & 0 & v + \Z^k\\
    0 & X_{n-k} & \Z^{n-k}\\
    0 & 0 & 1
\end{bmatrix}
\]
where $I_k$ is the identity matrix of degree $k$, and $X_{n-k} \in \GL(n-k,\Z)$ is a matrix of order $p$ (see \cite{Ch86}*{Theorem IV.5.1 and the beginning of Section IV.6}). Moreover, $v \in \frac{1}{p}\Z^{k}$, hence
\[
\overline\varphi(a^i) = i \cdot v + \Z^n
\]
and $\overline\varphi(\Gamma/\Z^n)$ is a group.
\end{example}

\begin{proposition}
\label{proposition:fixed-multiplicative-group}
Let $(A,+,\circ)$ be a cofinite integral brace of dimension $n$. Assume that $\Gamma \subset \GL(n,\Z) \ltimes \Q^n$ is a crystallographic group isomorphic to $(A,\circ)$. Then there exists $x \in \Q^n$ such that the pair $(\Gamma^x, \varphi)$ is a left brace isomorphic to $A$, where
\[
\Gamma^x := (1,x)^{-1} \Gamma (1,x) = (1,-x) \Gamma (1,x).
\]
\end{proposition}
\begin{proof}
Let $\Lambda$ be a map defined in Corollary \ref{corollary:representation-of-cib}. Then the pair $(\Lambda(A), \varphi)$ defines a cofinite integral brace isomorphic to $A$. Since the multiplicative group of $\Lambda(A)$ and $\Gamma$ are isomorphic subgroups of $\Aff(\Q^n)$, there exists an affine map $(X,x) \in \Aff(\Q^n)$ such that
\[
(X,x) \Lambda(A) (X,x)^{-1} = \Gamma.
\]
Since $(1,-x)(X,x) = (X,0)$, we get
\[
\Gamma^x = (X,0) \Lambda(A) (X,0)^{-1}.
\]
By construction, we can think of $(A,+)$ as a subgroup of $\Q^n$ such that $\Soc(A) = \Z^n$ and we get that
\[
\Lambda(A) = \{ (\lambda_a, a) : a \in A \}.
\]
We claim that the function $F \colon \Lambda(A) \to \Gamma^x$ given by
\[
F(\lambda_a,a) := (X,0)(\lambda_a,a)(X,0)^{-1} = (X \lambda_a X^{-1}, Xa)
\]
is a brace isomorphism. We have
\[
\varphi(\Gamma^x) = \{ \varphi(X \lambda_a X^{-1}, Xa) : a \in A \} = \{ Xa : a \in A \} = XA \subset \Q^n
\]
and since $X$ is an invertible matrix, $XA$ is a group. Moreover, it is easy to see that the map is one-to-one on $\Gamma^x$. Hence, the pair $(\Gamma^x,\varphi)$ defines a left brace. By definition, $F$ is an isomorphism of multiplicative groups of the braces. Now, for $a,b \in A$ we have
\[
\varphi F ( (\lambda_a,a) + (\lambda_b,b) ) = \varphi F( \lambda_{a+b}, a+b ) = X(a+b) = Xa + Xb = \varphi F(\lambda_a,a) + \varphi F (\lambda_b,b).
\]
This proves our claim.
\end{proof}

We end this section with the analogue of one of the Bieberbach theorems for cofinite integral braces.

\begin{theorem}\label{T:Bieberbach}
There exist only finitely many isomorphism classes of cofinite integral braces of dimension $n$.
\end{theorem}

\begin{proof}
By the Bieberbach theorem \cite{Sz25}*{Theorem 2.4.3}, there exist only finitely many isomorphism classes of crystallographic groups of dimension $n$. Hence, it is enough to prove that there exist only finitely many isomorphism classes of cofinite integral braces with fixed, up to isomorphism, multiplicative group. Let $\Gamma \subset \Aff_\Z(\Q^n)$ be a crystallographic group of dimension $n$, defined by short exact sequence \eqref{equation:ses:crystallographic}. By Proposition \ref{proposition:fixed-multiplicative-group}, up to isomorphism, every cofinite integral brace with multiplicative group isomorphic to $\Gamma$, is of the form $\Gamma^x$ for some $x \in {\Q}^n$. By Corollaries \ref{corollary:representation-of-cib} and \ref{corollary:translational-part-image}, every such a brace fits into the sequence
\[
\Soc(\Gamma^x) = \Z^n \subset \Gamma^x \subset G \ltimes \frac{1}{|G|} \Z^n.
\]
In particular, it is uniquely determined by the image
\[
\overline\varphi (\Gamma/\Z^n ) \subset \frac{1}{|G|} \Z^n / \Z^n,
\]
for which we only have a finite number of possibilities.
\end{proof}

\section{Classification of cofinite integral braces}\label{S:classification}

We start with a description of isomorphisms between cofinite integral braces.

Note that the map $\varphi$ is as in the previous section (see Remark \ref{remark:bijective-cocycle}).

\begin{proposition}
\label{proposition:isomorphism-by-translation}
Let $\Gamma_1, \Gamma_2 \subset \Aff_\Z(\Q^n)$ be crystallographic groups of dimension $n$. Let the pairs $(\Gamma_1, \varphi), (\Gamma_2, \varphi)$ define cofinite integral braces. Assume that $F \colon \Gamma_1 \to \Gamma_2$ is defined by the conjugation by $(1,x) \in \Aff({\Q}^n)$:
\[
\forall_{(B,b) \in \Gamma_1} F(B,b) := (1,x)(B,b)(1,x)^{-1}.
\]
If $F$ is an isomorphism of cofinite integral braces, then $F=id$.
\end{proposition}
\begin{proof}
As in the proof of Proposition \ref{proposition:fixed-multiplicative-group}, using Corollary \ref{corollary:representation-of-cib}, we can write
\[
\Gamma_1 = \{ (\lambda_a, a) : a \in \varphi(\Gamma_1) \}.
\]
For every $a,b \in \varphi(\Gamma_1)$ we have
\[
\begin{split}
\varphi F \varphi^{-1} (a+b) &= \varphi F (\lambda_{a+b}, a+b) \\
&= \varphi ((1,x)(\lambda_{a+b}, a+b)(1,-x)) \\
&= \varphi( \lambda_{a+b}, (1-\lambda_{a+b})x + a + b) = (1-\lambda_{a+b})x + a + b
\end{split}
\]
and
\[
\varphi F \varphi^{-1} (a) + \varphi F \varphi^{-1} (b) = (1-\lambda_{a})x + a + (1-\lambda_{b})x + b.
\]
Assuming $F$ is a brace homomorphism, we get
\[
(1-\lambda_{a+b})x + a + b = (1-\lambda_{a})x + a + (1-\lambda_{b})x + b,
\]
hence
\[
(1+\lambda_{a+b})x = ( \lambda_{a} + \lambda_{b} ) x.
\]
Using induction, we get that
\[
(d-1+\lambda_{da})x = d\lambda_{a}x
\]
holds for every natural number $d$. Take $d$ such that $da \in \Soc(\Gamma_1)$. Then $\lambda_{da} = 1$ and
\[
dx = (d-1 + \lambda_{da})x = d \lambda_a x
\]
and hence $x = \lambda_a x$. Finally, we get
\[
F(\lambda_a,a) = (\lambda_a, (1-\lambda_a)x + a) = (\lambda_a,a)
\]
and the result follows.
\end{proof}

\begin{corollary}
\label{corollary:isomorphism-of-cibs}
Let $\Gamma_1, \Gamma_2$ be as in Proposition \ref{proposition:isomorphism-by-translation}. Assume that $F \colon \Gamma_1 \to \Gamma_2$ is defined by the conjugation by $(X,x) \in \Aff({\Q}^n)$:
\[
\forall_{(B,b) \in \Gamma_1} F(B,b) := (X,x)(B,b)(X,x)^{-1}.
\]
If $F$ is an isomorphism of cofinite integral braces, then
\[
F(B,b) := (X,0)(B,b)(X,0)^{-1}.
\]
for every $(B,b) \in \Gamma_1$.
\end{corollary}

Before stating the main result of this section, let us recall some standard facts from the theory of crystallographic groups. Assume that $\Gamma \subset \Aff_\Z(\Q^n)$ is one such group with the maximal normal abelian subgroup $\Z^n$. Let it be defined by extension \eqref{equation:ses:crystallographic} that corresponds to the cohomology class
\[
\gamma \in H^1(G,\Q^n/\Z^n) \cong H^2(G,\Z^n).
\]
\begin{definition}[\cite{HP89}*{page 102}]
Let the function $s \colon G \to \Q^n$ be given by
\[
s(g) = x
\]
if $(g,x) \in \Gamma$.
We call $s$ a \emph{vector system} of $\Gamma$.
\end{definition}

An easy calculation shows that the map $\overline{s} \colon G \to \Q^n/\Z^n$, given by
\[
\overline{s}(g) = s(g) + \Z^n
\]
is a $1$-cocycle which is a representative of the cohomology class $\gamma$. Moreover, we have
\[
\Gamma/\Z^n = \{ (g, \overline{s}(g)) : g \in G \}.
\]

Note that there is an action of the normalizer $N$ of $G$ in $\GL(n,\Z)$ on the cohomology group $H^1(G,\Q^n/\Z^n)$, given by
\[
n * [c] = [n*c],
\]
where $[c]$ and $[n*c]$ are the cohomology classes of cocycles $c$ and $n*c$, respectively. The action of $N$ on the group of $1$-cocycles is given by the formula
\[
n*c(g) = n c(n^{-1}gn).
\]
We denote by $N_\gamma$ the stabilizer of the class $\gamma$ under the action $*$.

We state the main result of the paper.
\begin{theorem}
\label{theorem:classification}
Let $\Gamma \subset \Aff_\Z(\Q^n)$ be an $n$-dimensional crystallographic group, which fits into the short exact sequence \eqref{equation:ses:crystallographic}. Assume that the cohomology class $\gamma \in H^1(G,\Q^n/\Z^n)$ corresponds to this extension. Let $\mathcal S$ be the set of all one-to-one cocycles $\overline{s}$ from the class $\gamma$ for which $\overline{s}(G)$ is a group. Then the isomorphism classes of cofinite integral braces for which $\Gamma$ is the multiplicative group are in the bijection with the orbits of the action of $N_\gamma$ on the elements of $\mathcal S$.
\end{theorem}

\begin{proof}
By Proposition \ref{proposition:fixed-multiplicative-group} all cofinite integral braces with multiplicative group $\Gamma$ are, up to isomorphism, elements of the set $\{ \Gamma^x : x \in {\Q}^n \}$. For every $x \in {\Q}^n$ denote by $s^x$ a vector system of $\Gamma^x$. Note that any choice of $s^x$ uniquely determines the $1$-cocycle $\overline{s^x}$. Moreover, every $1$-cocycle from the class $\gamma$ determines the group $\Gamma^x$. Hence, by Corollary \ref{corollary:cib-on-holonomy-group}, the set of pairs $(\Gamma^x,\varphi)$ that define cofinite integral braces is in bijection with $\mathcal S$.

Assume that the pair $(\Gamma, \varphi)$ defines a cofinite integral brace. Assume that for some $x \in {\Q}^n$ the brace given by $(\Gamma^x, \varphi)$ is isomorphic to $\Gamma$. If $s$ and $s^x$ are vector systems for $\Gamma$ and $\Gamma^x$, respectively, then -- without loss of generality -- we can assume that
\[
s^x(g) = s(g) + (g-1)x
\]
for every $g \in G$. This means that both cocycles $\overline{s}$ and $\overline{s^x}$ lie in the cohomology class $\gamma$, in particular, they both are the elements of the set $\mathcal S$.

By Corollary \ref{corollary:isomorphism-of-cibs} there exists $X \in \GL(n,{\Q})$ such that
\[
\Gamma^x = (X,0) \Gamma (X^{-1},0).
\]
There are two consequences of this fact. First is that $X \Z^n = \Z^n$, which means that $X \in \GL(n,\Z)$. The second -- $G = XGX^{-1}$, hence $X$ normalizes $G$, i.e., $X \in N_{\GL(n,\Z)}(G)$. For every $g \in G$ we have
\[
(X,0)(g,s(g))(X^{-1},0) = (XgX^{-1}, Xs(g) ) \in \Gamma^x,
\]
hence
\[
s^x(XgX^{-1}) - Xs(g) \in \Z^n
\]
and therefore $\overline{s^x} = X * \overline{s}$.
\end{proof}

\section{Computing cofinite integral brace structures}
\label{section:computer-calculations}

We begin by determining the maximal possible exponents of additive groups of finite left braces with a fixed multiplicative group. The following lemma provides the necessary restriction.

\begin{lemma}
\label{lemma:socle-restriction}
Let $B$ be a finite left brace with multiplicative group $G$ and additive group $A$. Then $\Soc(B)$ is an abelian normal subgroup of $G$, it is isomorphic to a subgroup of $A$, and
\[
G/\Soc(B)
\]
is isomorphic to a subgroup of $\Aut(A)$.
\end{lemma}

\begin{proof}
Consider the homomorphism
\[
\lambda \colon G \to \Aut(A)
\]
associated with the brace structure. By definition, $\Soc(B)$ is a normal subgroup of $G$. Moreover, in $\Soc(B)$ the multiplicative operation coincides with the additive one, hence $\Soc(B)$ is isomorphic to an abelian subgroup of both $G$ and $A$. Finally, by the first isomorphism theorem,
\[
G/\Soc(B)\cong \lambda(G)\leq \Aut(A).
\]
\end{proof}

Lemma~\ref{lemma:socle-restriction} yields a practical necessary condition on the additive group of a finite left brace with multiplicative group $G$. Indeed, if $A$ is the additive group of such a brace, then one must find an abelian normal subgroup $N\trianglelefteq G$ that can occur as the socle, with $N$ embedding into $A$ and with the quotient $G/N$ acting faithfully on $A$. This observation leads to a finite search procedure for the maximal possible exponents of the additive group, presented in Algorithm \ref{alg:maximal-exponents}. For brevity, we use the following notation:

\begin{itemize}
\item $\mathcal{D}(o)$ -- the sequence of the positive divisors of the natural number $o$ in decreasing order;
\item $\mathcal{A}(o,e)$ -- the set of representatives of isomorphism classes of abelian groups of order $o$ and exponent $e$.
\end{itemize}

\begin{algorithm}

\SetKwFunction{MaxExp}{MaximalExponentsByMultiplicativeGroup}

\fn{\MaxExp{$G$}}{

    \If{$G$ is not solvable}{
        \Ret{$\emptyset$}\;
    }
    
    $E:=\emptyset$\;
    
    \For{$e \in \mathcal{D}(|G|)$\nllabel{alg:maximal-exponents:outer-loop}}{
        \If{$e$ divides some element of $E$}{
            \algkw{go to line~\ref{alg:maximal-exponents:outer-loop}}\;
        }
        \For{$A \in \mathcal{A}(|G|,e)$}{
            \For{every abelian normal subgroup $N\trianglelefteq G$}{
                \If{$N$ embeds into $A$ and $G/N$ embeds into $\Aut(A)$}{
                    $E:=E\cup\{e\}$\;
                    \algkw{go to line~\ref{alg:maximal-exponents:outer-loop}}\;
                }
            }
        }
    }
    \Ret{$E$}\;
}
\caption{Finding maximal possible exponents of additive groups of finite left braces with multiplicative group $G$}
\label{alg:maximal-exponents}
\end{algorithm}

\begin{lemma}
Let $w \colon G \to \Q^n/\Z^n$ be a $1$-coboundary. There exists $x \in \Q^n$ such that
\[
w(g) = (g-1)x + \Z^n
\]
for every $g \in G$.
\end{lemma}
\begin{proof}
Take any $x \in \Q^n$ such that $w = \delta^1(x+\Z^n)$. Then, for every $g \in G$, we have
\[
w(g) = (g-1)(x+\Z^n) = (g-1)x + \Z^n.
\]
\end{proof}

\begin{corollary}
\label{corollary:coboundary-basis}
Let $G \subset \GL(n,\Z)$ be a finite group. Define maps $b_1,\ldots,b_n \colon G \to \Z^n$ by
\[
b_i(g) = (g-1)^{(i)},
\]
where $g \in G$ and $(g-1)^{(i)}$ denotes the $i$th column of the matrix $g-1$. Let $w \colon G \to \Q^n/\Z^n$ be a $1$-coboundary. Then
\[
w(g) = \sum_{i=1}^n \alpha_i b_i(g) + \Z^n
\]
for every $g \in G$, where $\alpha_1,\ldots,\alpha_n \in \Q$.
\end{corollary}
\begin{proof}
Denote by $e_i$ the $i$th column of the identity matrix of degree $n$. Let
\[
x = \alpha_1 e_1 + \ldots + \alpha_n e_n \in \Q^n,
\]
be a vector such that $w = \delta^1(x+\Z^n)$. Obviously, $\alpha_1, \ldots, \alpha_n \in \Q$ and for every $g \in G$ we have
\[
w(g) = (g-1)x + \Z^n = \sum_{i=1}^n \alpha_i (g-1)e_i + \Z^n = \sum_{i=1}^n \alpha_i b_i(g) + \Z^n.
\]
\end{proof}

From the above proof, one can simply read off that $b_1,\ldots,b_n$ generate the vector space $B^1(G,\Q^n)$ of $1$-coboundaries. We want to work with its subgroup, which is closely related to $B^1(G,\Z^n)$, but in general, can be a bit larger.
\begin{lemma} The set
$\overline{B^1}(G,\Z^n) := \{ f \colon G \to \Z^n : f \in B^1(G,\Q^n) \}$
of $1$-coboundaries that take value in $\Z^n$ is a group and $B^1(G,\Z^n)$ is its finite index subgroup.
\end{lemma}
\begin{proof}
It is easy to check that $\overline{B^1}(G,\Z^n)$ is a subgroup of $B^1(G,\Q^n)$, hence if $r = \dim_{\Q}(B^1(G,\Q^n)) \leq n$, we get that it is a free abelian group of rank $r$. Since $b_1,\ldots,b_n$ generate the space $B^1(G,\Q^n)$, for any set $\overline{b_1},\ldots,\overline{b_r}$ of generators of $\overline{B^1}(G,\Z^n)$ there exists an integer $e$ such that $e\overline{b}_1,\ldots,e\overline{b}_r \in B^1(G,\Z^n)$.
\end{proof}

\begin{remark}
\label{remark:finding-basis}
Finding the rank $r$ and the basis $\overline{b}_1, \ldots, \overline{b}_r$ of $\overline{B^1}(G,\Z^n)$ uses the invariant factor theorem, see \cite{CR62}*{Theorem 16.8} for example. If one puts the elements of the group $G$ in a sequence, say $(g_1, \ldots, g_l)$, then the columns of the matrix
\[
B = \begin{bmatrix}
    b_1(g_1) & \ldots & b_n(g_1) \\
    \vdots   & \ddots & \vdots   \\
    b_1(g_l) & \ldots & b_n(g_l)
\end{bmatrix}
=
\begin{bmatrix}
    g_1 - 1\\
    \vdots \\
    g_l - 1
\end{bmatrix}
\]
generate a subgroup of $\Z^{nl}$, isomorphic to $B^1(G,\Z^n)$. Let $S = LBR$ be the Smith normal form of $B$, with row and column transformation matrices $L$ and $R$, respectively. A coboundary $b = \delta^1x$, where $x \in \Q^n$, is an element of $\overline{B^1}(G,\Z^n)$ if and only if $Bx \in \Z^{nl}$. We get that $v = LBx = SR^{-1}x \in \Z^{nl}$. Moreover, if $r = \dim_{\Q}(B^1(G,\Q^n)) = \operatorname{rk}_\Z(\overline{B^1}(G,\Z^n))$, then only $r$ first coordinates of $v$ may be non-zero and $Bx = L^{-1}v$ is a $\Z$-linear combination of the first $r$ columns of $L^{-1}$. Now, if the invariant factors of $B$ are $\delta_1, \ldots, \delta_r$, the equality $L^{-1}S = BR$ shows that for $1 \leq i \leq r$, the $i$th column of $L^{-1}$ equals $BR(\frac{1}{\delta_i}e_i)$, where $e_i$ is the $i$th column of the identity matrix of degree $n$. Hence, every such column defines an element $\overline{b}_i \in \overline{B^1}(G,\Z^n)$. The elements $\overline{b}_1,\ldots, \overline{b}_r$ are linearly independent set of generators of $\overline{B^1}(G,\Z^n)$, therefore they form its basis.
\end{remark}

The following proposition lays a ground for Algorithm \ref{alg:finding-cibs-with-exponent}. %

\begin{proposition}
Let $G \subset \GL(n,\Z)$ be a finite group. Let $s \colon G \to \Q^n$ be a vector system, i.e., a map such that the composition
\[
\begin{tikzcd}[column sep=large]
\overline s \colon G \arrow[r,"s"] & \Q^n \arrow[r,"\bmod{\Z^n}"] & \Q^n/\Z^n
\end{tikzcd}
\]
is a $1$-cocycle. Assume that $e \in \Z$ is such that $e \cdot s(G) \subset \Z^n$ and that for some $x \in \Q^n$, $\overline{s^x}(G) \subset \Q^n/\Z^n$ is a group of the exponent that divides $e$. Then there exist integers $\beta_1,\ldots,\beta_r$ such that
\[
s^x - s = \sum_{i=1}^r \frac{\beta_i}{e} \overline{b}_i,
\]
where $\{\overline{b}_1,\ldots,\overline{b}_r\}$ is a basis of $\overline{B^1}(G,\Z^n)$.
\end{proposition}
\begin{proof}
Since $s^x-s$ is a $1$-coboundary, there exist rational numbers $\alpha_1,\ldots,\alpha_r$ such that
\[
s^x-s = \sum_{i=1}^r \alpha_i \overline{b}_i.
\]
By assumption, for every $g \in G$ we have that
\begin{equation}
\label{equation:coboundaries}
e \sum_{i=1}^r \alpha_i \overline{b}_i(g) \in \Z^n.
\end{equation}
Let $G = \{g_1,\ldots,g_l\}$ and 
\[
B = \begin{bmatrix}
\overline{b}_1(g_1) & \ldots & \overline{b}_r(g_1)\\
\vdots & \ddots & \vdots\\
\overline{b}_1(g_l) & \ldots & \overline{b}_r(g_l)\\
\end{bmatrix}.
\]
Let $S$ be the Smith normal form of $B$ and let $S = LBR$ for unimodular matrices $L$ and $R$. Since the maps $\overline{b}_1, \ldots, \overline{b}_r$ form a basis of $\overline{B^1}(G,\Z^n)$, we get that
\[
S = \begin{bmatrix}
    I\\0\\ \vdots \\0
\end{bmatrix},
\]
where $I$ is the identity matrix of degree $r$.

Let
\[
C = e \begin{bmatrix}
\alpha_1 \\ \vdots \\ \alpha_r
\end{bmatrix}.
\]
Condition \eqref{equation:coboundaries} translates to
\[
BC = L^{-1}SR^{-1}C \in \Z^{nl}.
\]
Now, because $L$ is unimodular, the form of $S$ implies that $R^{-1}C \in \Z^r$, hence $C$ itself lies in $\Z^r$. It is enough to take
\[
\beta_i = e \alpha_i
\]
for $1 \leq i \leq r$.
\end{proof}

\begin{algorithm}
\KwIn{a vector system $s \colon G \to \Q^n$, an integer $e$ such that $e \cdot s(G) \subset \Z^n$}
\KwOut{the set $S$ of vector systems of the form $s^x$ that correspond to cofinite integral braces and such that the exponent of $(s^x(G)+\Z^n)/\Z^n$ divides $e$}

\SetKwFunction{FindCIBs}{FindCIBsWithExponent}

\renewcommand{\baselinestretch}{1.15}\selectfont

\fn{\FindCIBs{$s,e$}}{
    $S := \emptyset$\;
    calculate $\overline{b}_1,\ldots,\overline{b}_r$ as in Remark~\ref{remark:finding-basis}\;
    \For{$(\alpha_1,\ldots,\alpha_r) \in \{0,\ldots,e-1\}^r$}{
        $s' := s + \sum_{i=1}^r \frac{\alpha_i}{e}\, \overline{b}_i$\;
        \If{$(s'(G)+\Z^n)/\Z^n$ is a subgroup of $\Q^n/\Z^n$ of order $|G|$}{
            $S := S \cup \{s'\}$\;
        }
    }
    \Ret{$S$}\;
}
\caption{Finding cofinite integral brace structures on a given crystallographic group}
\label{alg:finding-cibs-with-exponent}
\end{algorithm}

In practice, there may be one problem with Algorithm \ref{alg:finding-cibs-with-exponent} -- its cost grows quickly with the dimension and the exponent bound. The next lemma provides a useful reduction step: when the multiplicative group has composite order, one may restrict to a suitable proper $p$-subgroup and thereby reduce the search to smaller groups.

\begin{lemma}\label{L:reduction}
Let $G$ be a group of composite order, $s \colon G \to \Q^n$ be a vector system of a cofinite integral brace, and let $p$ be a prime divisor of $|G|$. Then there exists a proper $p$-subgroup $P \leq G$ such that the restricted map
\[
s|_P \colon P \to \Q^n
\]
is a vector system of a cofinite integral brace.
\end{lemma}
\begin{proof}
By \cite{CV26}*{Lemma 12.2.10}, if $G$ is not a $p$-group, then $P$ is one of the $p$-Sylow subgroups of $G$. On the other hand, if $G$ is a $p$-group then by \cite{CV26}*{Corollary 13.1.4}, $P$ may be found in the set of maximal normal subgroups of $G$.
\end{proof}

\begin{remark}
Although the above lemma provides only a necessary condition for the existence of a cofinite integral brace structure on a crystallographic group, this condition can in practice be turned into a sufficient one.

Let $G$ be a group and let $s \colon G \to \Q^n$ be a vector system. Suppose that, for some proper subgroup $P \leq G$, we have found a vector system $s' \colon P \to \Q^n$ such that:
\begin{enumerate}
\renewcommand{\labelenumi}{\arabic{enumi})}
\item $s'$ corresponds to a cofinite integral brace;
\item after composing with the natural map $\Q^n \to \Q^n/\Z^n$, the two cocycles $\overline{s'}$ and $\overline{s}_{|P}$ lie in the same cohomology class in $H^1(P,\Q^n/\Z^n)$.
\end{enumerate}
Then, once an exponent $e$ is fixed (for instance, one may take $e = |G|$), in Algorithm~\ref{alg:finding-cibs-with-exponent} it is enough to consider only those tuples $(\alpha_1,\ldots,\alpha_r)$ for which
\[
s' = \left(s + \frac{\alpha_1}{e} \overline{b}_1 + \cdots + \frac{\alpha_r}{e} \overline{b}_r\right)_{|P}.
\]
Equivalently, the coefficients $\alpha_1/e,\ldots,\alpha_r/e$ can be determined effectively by solving over $\Q/\Z$ a system of equations of the form
\[
\begin{bmatrix}
p_1 - 1\\
\vdots\\
p_k - 1
\end{bmatrix} \cdot x
=
\begin{bmatrix}
s'(p_1)-s(p_1)\\
\vdots\\
s'(p_k)-s(p_k)
\end{bmatrix},
\]
where $P=\{p_1,\ldots,p_k\}$. This method is described, for example, in \cite{Lu13}.
\end{remark}

\section{Example: Hantzsche-Wendt group}\label{S:GHW}
In this section we start with structural properties of cofinite integral braces with the multiplicative group being generalized Hantzsche-Wendt groups. By Lemma \ref{lemma:hantzsche-wendt-existence} we know that every such group admits a cofinite integral brace structure. At the end we present the example of classification of CIB structures on the three-dimensional Hantzsche-Wendt group, i.e, the minimal non-abelian example of a GHW group lying in $\operatorname{SL}(n,\Z)\ltimes\Q^n$, also known in the literature as the Promislow group.

\begin{proposition}
\label{proposition:exponent-in-hw-groups}
Let $\Gamma$ be a generalized Hantzsche-Wendt group, defined by the short exact sequence \eqref{equation:ses:crystallographic}. Assume that the pair $(\Gamma,\varphi)$ defines a cofinite integral brace. Then the exponent of the additive group of the quotient brace $\Gamma/\Z^n$ divides $4$.
\end{proposition}
\begin{proof}
By Lemma \ref{lemma:hantzsche-wendt-existence}, we can assume that $\Gamma \subset \mathcal{D} \ltimes \frac{1}{2}\Z^n$. Let $s \colon G \to \frac{1}{2}\Z^n$ be the vector system of $\Gamma$. By the proof of Theorem \ref{theorem:classification}, any cofinite integral brace with the multiplicative group $\Gamma$ has a vector system of the form $s^x$ such that the image of the cocycle $\overline s^x \colon G \to \Q^n/\Z^n$ defined by $s^x$ is a group. Let $(B,b) \in \Gamma$ be such that $B \neq 1$. This means that
\[
B = \diag(B_1, \ldots, B_n) \text{ and } b = (b_1, \ldots, b_n)^T,
\]
where $B_i = \pm 1$ and $b_i \in \frac{1}{2}\Z$, for $1 \leq i \leq n$. Let $x = (x_1,\ldots,x_n) \in \R^n$. The $i$th coordinate of $s^x(B)$ equals
\[
(1-B_i)x_i + b_i \in \{ b_i, 2x_i + b_i \}.
\]
The order of $(B,b) + \Z^n$ in the additive group is finite, hence there exists the smallest positive integer $k$ such that
\[
k b_i \in \Z \text{ or } k (2x_i+b_i) \in \Z.
\]
In particular, since $G$ is a $2$-group, $k$ is a power of $2$. Assume that $k \geq 8$. Then $B_i=-1$, $kb_i \in \Z$ and
\[
k (2x_i+b_i) \in \Z \Leftrightarrow k 2x_i \in \Z.
\]
Denote by $p^{(i)} \colon \Q^n/\Z^n \to \Q/\Z$ the projection to the $i$th coordinate. We can see that the elements of $p^{(i)}\overline{s}^x(G)$ are either of order $1,2$ or $k$, hence $\overline{s}^x(G)$ is not a group.
\end{proof}

\begin{theorem}\label{T:3dHW}
There are three isomorphism classes of cofinite integral braces for which the multiplicative group is the three-dimensional Hantzsche-Wendt group.
\end{theorem}

\begin{proof}
Let $a$ and $b$ be generators of the holonomy group $G$ of the three-dimensional Hantzsche-Wendt group $\Gamma$. One can take
\[
a = \begin{bmatrix}
1 & 0 & 0\\
0 &-1 & 0\\
0 & 0 &-1\\
\end{bmatrix}
\text{ and }
b = \begin{bmatrix}
-1 & 0 & 0\\
 0 & 1 & 0\\
 0 & 0 &-1\\
\end{bmatrix}.
\]
The vector system of $\Gamma$ that defines the cohomology class $\gamma \in H^1(G,\Q^n/\Z^n)$ can be given as follows
\begin{equation}
\label{equation:hw-vector-system}
s(a) = \begin{bmatrix}
\frac{1}{2}\\
\frac{1}{2}\\
0
\end{bmatrix},
s(b) = \begin{bmatrix}
0\\
\frac{1}{2}\\
\frac{1}{2}\\
\end{bmatrix},
s(ab) = \begin{bmatrix}
\frac{1}{2}\\
0\\
\frac{1}{2}\\
\end{bmatrix},
s(1) = \begin{bmatrix}
0\\
0\\
0
\end{bmatrix}.
\end{equation}
The stabilizer $N_\gamma$ equals $O(3,\Z)$ and is generated by the diagonal and permutation matrices of $\GL(3,\Z)$. By Proposition \ref{proposition:exponent-in-hw-groups}, for every $x \in \Q^3$ such that the vector system $s^x$ defines a cofinite integral brace we have $s^x(G) \in \frac{1}{4}\Z^3$. In this case we can simply take all denominators to be $4$ and add numerators modulo $4$. Additionally, the neutral element of $G$ is mapped to the zero vector. Hence, it is enough to write down values of non-trivial elements of $G$. We do it in the row convention, i.e., we transpose the values, and we use the sequence $s(a),s(b),s(ab)$. For example, the vector system \eqref{equation:hw-vector-system} is encoded as
\[
\begin{bmatrix}
    2 & 2 & 0\\
    0 & 2 & 2\\
    2 & 0 & 2
\end{bmatrix}.
\]
Now, the action of the permutation matrix on $\overline{s^x}$ is precisely the action of the corresponding permutation on the matrix. The action of a diagonal matrix is simply multiplying the matrix from the right.

In the case of our group, using methods described in Section \ref{section:computer-calculations}, we get $14$ vector systems that define cofinite integral braces. The orbits of the corresponding matrices 
\[
A = \begin{bmatrix}
    2 & 2 & 0\\
    0 & 2 & 2\\
    2 & 0 & 2
\end{bmatrix},
B = \begin{bmatrix}
    2 & 2 & 2\\
    2 & 2 & 0\\
    0 & 0 & 2
\end{bmatrix}
\text{ and }
C = \begin{bmatrix}
    2 & 2 & 3\\
    2 & 2 & 1\\
    0 & 0 & 2
\end{bmatrix}
\]
under the action of $N_\gamma$ are presented in Diagrams \ref{diagram:orbit-1}, \ref{diagram:orbit-2} and \ref{diagram:orbit-3}. Since the orbits have 14 elements in total, the result follows.
\end{proof}

\begin{remark}
The stabilizers of $A,B$ and $C$ under the action of $N_\gamma$ are as follows:
\begin{enumerate}
    \item $\operatorname{Stab}_{N_\gamma}(A) = \langle [\pm 1, \pm 1, \pm 1], (1\,2\,3) \rangle \cong C_2^3 \rtimes A_3 = C_2 \wr A_3$.
    \item $\operatorname{Stab}_{N_\gamma}(B) = \langle [\pm 1, \pm 1, \pm 1] \rangle \cong C_2^3$.
    \item $\operatorname{Stab}_{N_\gamma}(C) = \langle [\pm 1, \pm 1, 1], [1,1,-1](1\,2) \rangle \cong D_8$.
\end{enumerate}
Mark that we identify permutations with their matrices, and that $[a,b,c]$ encodes the diagonal matrix
\[
\begin{bmatrix}
a & 0 & 0\\
0 & b & 0\\
0 & 0 & c
\end{bmatrix}.
\]
\end{remark}

\begin{fdiagram}[ht]
\[
\begin{tikzcd}[ampersand replacement=\&, column sep=large]
\begin{bmatrix}
    2 & 2 & 0\\
    0 & 2 & 2\\
    2 & 0 & 2
\end{bmatrix}
\arrow[leftrightarrow]{r}[above]{(1\,2)}[below]{(1\,3)} \&
\begin{bmatrix}
    2 & 0 & 2\\
    2 & 2 & 0\\
    0 & 2 & 2
\end{bmatrix}
\end{tikzcd}
\]
\caption{Cofinite integral braces on three-dimensional Hantzsche-Wendt group. Orbit of $A$.}
\label{diagram:orbit-1}
\end{fdiagram}

\begin{fdiagram}[ht]
\[
\begin{tikzcd}[ampersand replacement=\&, column sep=large]
\begin{bmatrix}
    2 & 2 & 2\\
    2 & 2 & 0\\
    0 & 0 & 2
\end{bmatrix}
\arrow[leftrightarrow]{r}{(1\,2)}
\arrow[leftrightarrow]{d}{(2\,3)}\&
\begin{bmatrix}
    2 & 2 & 0\\
    2 & 2 & 2\\
    0 & 0 & 2
\end{bmatrix}
\arrow[leftrightarrow]{r}{(2\,3)}
\arrow[leftrightarrow]{d}{(1\,3)}\&
\begin{bmatrix}
    2 & 0 & 2\\
    0 & 2 & 0\\
    2 & 2 & 2
\end{bmatrix}
\arrow[d,leftrightarrow, "(1\,2)"] \\
    \begin{bmatrix}
    2 & 2 & 2\\
    0 & 2 & 0\\
    2 & 0 & 2
\end{bmatrix}
\arrow[r,leftrightarrow, "(1\,2)"]
\&
    \begin{bmatrix}
    2 & 0 & 0\\
    2 & 2 & 2\\
    0 & 2 & 2
\end{bmatrix}
\arrow[r,leftrightarrow, "(2\,3)"]
\&
    \begin{bmatrix}
    2 & 0 & 0\\
    0 & 2 & 2\\
    2 & 2 & 2
\end{bmatrix}
\end{tikzcd}
\]
\caption{Cofinite integral braces on three-dimensional Hantzsche-Wendt group. Orbit of $B$.}
\label{diagram:orbit-2}
\end{fdiagram}

\begin{fdiagram}[h!t]
\[
\begin{tikzcd}[ampersand replacement=\&, column sep=huge]
    \begin{bmatrix}
    2 & 2 & 3\\
    2 & 2 & 1\\
    0 & 0 & 2
\end{bmatrix}
\arrow[leftrightarrow]{r}[above]{(1\,2)}[below]{[1,1,-1]}
\arrow[d,leftrightarrow, "(2\,3)"]\&
\begin{bmatrix}
    2 & 2 & 1\\
    2 & 2 & 3\\
    0 & 0 & 2
\end{bmatrix}
\arrow[r,leftrightarrow, "(2\,3)"]
\arrow[d,leftrightarrow, "(1\,3)"]\&
\begin{bmatrix}
    2 & 1 & 2\\
    0 & 2 & 0\\
    2 & 3 & 2
\end{bmatrix}
\arrow[d,leftrightarrow, "(1\,2)"] \\
    \begin{bmatrix}
    2 & 3 & 2\\
    0 & 2 & 0\\
    2 & 1 & 2
\end{bmatrix}
\arrow[r,leftrightarrow, "(1\,2)"]
\&
    \begin{bmatrix}
    2 & 0 & 0\\
    3 & 2 & 2\\
    1 & 2 & 2
\end{bmatrix}
\arrow[r,leftrightarrow, "(2\,3)"]
\&
    \begin{bmatrix}
    2 & 0 & 0\\
    1 & 2 & 2\\
    3 & 2 & 2
\end{bmatrix}
\end{tikzcd}
\]
\caption{Cofinite integral braces on three-dimensional Hantzsche-Wendt group. Orbit of $C$.}
\label{diagram:orbit-3}
\end{fdiagram}

\section{Results in low dimensions}\label{S:low_dim}

This section describes the calculations on crystallographic groups up to dimension $6$. These were enumerated with the usage of the CARAT computer package \cite{CARAT}. Calculation of CIB structures on these groups was obtained using GAP \cite{GAP4}. The code needed to reproduce the results can be found in the repository \cite{Lu26git}.

The basic goal was to calculate the number of structures of cofinite integral braces, up to isomorphism, on those groups. This is presented in Table \ref{table:cibs-on-crystallographic-groups}. Table \ref{table:cibs-on-bieberbach-groups} gives the analogous information for Bieberbach groups. Having those structures enumerated, we were able to confront the two- and three-dimensional phenomena described by Rump. We will present this in the following series of remarks.

\begin{remark}
\label{remark:quotient-brace-rump}
The table at the end of \cite{Ru24}*{Section 6} presents the possible additive structures of $A/\Soc(A)$ for every cofinite integral brace of dimension $2$. In Table \ref{table:dimension-two-details} we give a more detailed view, i.e., we also show the number of cofinite integral brace structures on every wallpaper group. Note that there is one such structure that is not listed by Rump (the appropriate number is marked with an asterisk).
\end{remark}

\begin{remark}
An analogous approach, but for three-dimensional Bieberbach groups, is given in Table \ref{table:dimension-three-details}. Note that our approach detected 5 distinct cofinite integral brace structures on each of the groups $\mathcal B_3$ and $\mathcal B_4$. In comparison, Rump establishes the existence of a CIB on $\mathcal B_3$ in the proof of \cite{Ru24}*{Theorem 4} and identifies two non-isomorphic CIB structures on $\mathcal B_4$ in \cite{Ru24}*{Example 15}.%
\end{remark}

\begin{remark}\label{reducible}
Let $\Gamma$ be a crystallographic group given by the short exact sequence \eqref{equation:ses:crystallographic}. If the dimension of $\Gamma$ equals $2$, then by \cite{Ru24}*{Proposition 11} the reducibility of the action of $G$ on $L$ implies the existence of a cofinite integral brace structure of $\Gamma$. This is also true for $3$-dimensional torsion-free case: every such Bieberbach group admits a cofinite integral brace structure, while the above action is reducible. This is not the case for all crystallographic groups in general. Already in dimension $3$ there is an example of a crystallographic group, with the $G$ action the same as for the group $\mathcal G_5$, which does not have any cofinite integral brace structure. %
\end{remark}

\begin{remark}\label{asymmorphic}
Recall that a crystallographic group is called \emph{symmorphic} if it splits over its translation subgroup as a semidirect product $\mathbb{Z}^n \rtimes G$, and \emph{asymmorphic} otherwise. Rump remarks just before \cite{Ru24}*{Corollary on page 296} that if a wallpaper group does not have any cofinite integral brace structure then it is symmorphic. Since every non-trivial Bieberbach group is asymmorphic, Table \ref{table:cibs-on-bieberbach-groups} shows that this implication fails to hold already in dimension 4. The only two 4-dimensional Bieberbach groups without any CIB structure are denoted by CARAT symbol \texttt{min.20.1.4.3} and \texttt{group.47.1.3.6}. These correspond to \texttt{4/1/6/4} and \texttt{13/1/3/8} of \cite{BBNWZ}, respectively.
\end{remark}

\begin{table}[ht]
    \[
    \begin{array}{r|r|r|r|r}
        \multirow{2}{*}{\text{dim}} & \multicolumn{3}{c|}{\text{crystallographic groups}} & \multirow{2}{*}{\text{CIBs}}\\ \cline{2-4}
        & \text{total} & \text{with CIB} & \text{without CIB} & \\ \hline
        1 & 2 & 2 & 0 & 2\\
        2 & 17 & 12 & 5 & 16\\
        3 & 219 & 142 & 77 & 319\\
        4 & 4,783 & 3,341 & 1,442 & 15,888\\
        5 & 222,018 & 181,901 & 40,117 & 2,194,473\\
        6 & 28,927,915 & 26,466,343 & 2,461,572 & 927,879,728
    \end{array}
    \]
    \caption{Cofinite integral braces (CIBs) on crystallographic groups}
    \label{table:cibs-on-crystallographic-groups}
\end{table}

\begin{table}[ht]
    \[
    \begin{array}{r|r|r|r|r}
        \multirow{2}{*}{\text{dim}} & \multicolumn{3}{c|}{\text{Bieberbach groups}} & \multirow{2}{*}{\text{CIBs}}\\ \cline{2-4}
        & \text{total} & \text{with CIB} & \text{without CIB} & \\ \hline
        1 & 1 & 1 & 0 & 1\\
        2 & 2 & 2 & 0 & 3\\
        3 & 10 & 10 & 0 & 24\\
        4 & 74 & 72 & 2 & 376\\
        5 & 1,060 & 1,022 & 38 & 14,014\\
        6 & 38,746 & 37,602 & 1,144 & 1,383,718\\
    \end{array}
    \]
    \caption{Cofinite integral braces on Bieberbach groups}
    \label{table:cibs-on-bieberbach-groups}
\end{table}

\begin{table}[ht]
    \centering
    \begin{tabular}{l|l|c|cccccc}
       $\Gamma$     & CARAT         & $G$    & $1$ & $C_2$ & $C_3$ & $C_2^2$ & $C_4$ & $C_2 \times C_4$\\ \hline
       \textbf{p1}  & min.2.1.1.0   & $1$    & 1 \\
       \textbf{p2}  & group.1.1.1.0 & $C_2$  & & 1\\
       \textbf{pm}  & min.3.1.1.0   & $C_2$  & & 1\\
       \textbf{pg}  & min.3.1.1.1   & $C_2$  & & 2\\
       \textbf{pmm} & group.2.1.1.0 & $C_2^2$& & & & 1\\
       \textbf{pmg} & group.2.1.1.1 & $C_2^2$& & & & 1 & 1\\
       \textbf{pgg} & group.2.1.1.3 & $C_2^2$& & & & 2 & {1\textsuperscript{\!*}}\\
       \textbf{cm}  & min.3.1.2.0   & $C_2$  & & 1\\
       \textbf{cmm} & group.2.1.2.0 & $C_2^2$& & & & & 1\\
       \textbf{p4}  & min.4.1.1.0   & $C_4$  & & & & 1\\
       \textbf{p4g} & max.2.1.1.1   & $D_4$  & & & & & & 1\\
       \textbf{p3}  & min.5.1.1.0   & $C_3$  & & & 1\\
    \end{tabular}
    \caption{Number of cofinite integral braces in dimension $2$ with the given additive group on the quotient by the socle. $G$ describes the multiplicative group of $\Gamma/\Soc(\Gamma)$. For the description of the asterisk character see Remark \ref{remark:quotient-brace-rump}.}
    \label{table:dimension-two-details}
\end{table}

\begin{table}[ht]
    \centering
    \begin{tabular}{l|l|c|cccccc}
    $\Gamma$        & CARAT         & $G$    & $1$ & $C_2$ & $C_3$ & $C_2^2$ & $C_4$ & $C_6$\\ \hline
    $\mathcal{G}_1$ & min.6.1.1.0   & $1$    & 1\\
    $\mathcal{G}_2$ & min.7.1.1.1   & $C_2$  & & 2\\
    $\mathcal{G}_3$ & min.13.1.1.1  & $C_3$  & & & 2\\
    $\mathcal{G}_4$ & min.11.1.1.1  & $C_4$  & & & & & 2\\
    $\mathcal{G}_5$ & group.18.1.1.1& $C_6$  & & & & & & 1\\
    $\mathcal{G}_6$ & min.10.1.1.7  & $C_2^2$& & & & 2 & 1\\
    $\mathcal{B}_1$ & min.8.1.1.1   & $C_2$  & & 2\\
    $\mathcal{B}_2$ & min.8.1.2.1   & $C_2$  & & 1\\
    $\mathcal{B}_3$ & min.9.1.1.6   & $C_2^2$& & & & 4 & 1\\
    $\mathcal{B}_4$ & min.9.1.1.7   & $C_2^2$& & & & 4 & 1\\
    \end{tabular}
    \caption{Number of cofinite integral braces on $3$-dimensional Bieberbach groups, with the given additive group on the quotient by the socle. $G$ describes the multiplicative group of $\Gamma/\Soc(\Gamma)$.}
    \label{table:dimension-three-details}
\end{table}

\section{Concluding remarks and open problems}\label{S:conclusions}

The results of this paper suggest two natural directions for further research on
cofinite integral braces.
\subsection{Crystallographic and Bieberbach groups without CIB structures}
In dimension $2$, exactly $12$ of the $17$ wallpaper groups admit a CIB structure (Theorem~\ref{T:wallpaper_CIB}). All four asymmorphic wallpaper groups do, and the five groups without a CIB structure, $\mathbf{p4m}$, $\mathbf{p3m1}$, $\mathbf{p31m}$, $\mathbf{p6}$ and $\mathbf{p6m}$, are symmorphic. In dimension $3$, all ten Bieberbach groups admit CIB structures \cite{Ru24}*{Theorem~4}. However, in dimension $4$, there exist two Bieberbach groups that admit no CIB structure (see Remark~\ref{asymmorphic} and Table~\ref{table:cibs-on-bieberbach-groups}). Moreover, Rump's observation that every $\mathbb{Z}$-reducible wallpaper group admits a CIB structure \cite{Ru24}*{Proposition~11} does not extend to dimension $3$ (Remark~\ref{reducible}). Theorem~\ref{theorem:classification} gives an effective criterion for the existence of a CIB structure on a given crystallographic group, but not a structural one.
\begin{problem}
Find an intrinsic group-theoretic or geometric characterization of those Bieberbach (and crystallographic) groups that admit at least one cofinite integral brace structure.
\end{problem}
Although five wallpaper groups admit no CIB structure, every wallpaper group is the multiplicative group of some brace \cite{Ru24}*{Corollary, p.~296}. Rump's argument is the following. If the holonomy group $G$ of a symmorphic crystallographic group $\Gamma=\mathbb{Z}^n\rtimes G$ is the multiplicative group of a finite brace $B$, then $\Gamma$ is the multiplicative group of the semidirect product $B\ltimes\mathbb{Z}^n$ of $B$ with the trivial brace $\mathbb{Z}^n$. Since the additive group $B\oplus\mathbb{Z}^n$ of this brace has torsion, it is not a cofinite integral brace. In dimension two this construction covers the five
wallpaper groups without CIB structures, all of which are symmorphic. It does not apply to asymmorphic groups, in particular not to the two four-dimensional Bieberbach groups without CIB structures. 
\begin{problem}
Is every crystallographic group the multiplicative group of some brace? In particular, are the two four-dimensional Bieberbach groups without CIB structures multiplicative groups of braces?
\end{problem}
\subsection{Connection to structure groups of set-theoretic solutions}
By Example~\ref{E:structure_group}, every structure group $G(X, r)$ of a solution $(X,r)$ carries a canonical cofinite integral brace structure; following \cite{Ru24} we denote this brace by $A_X$. Its additive group is free abelian with basis $X$, and every $\lambda_a$ permutes $X$. Hence, the degree map
\[
  \varepsilon\colon A_X\to\Z,\qquad \sum_{x\in X} n_x x\longmapsto\sum_{x\in X} n_x ,
\]
is a brace morphism onto the trivial brace $\Z$. Its kernel
$I_X=\ker\varepsilon$, which is additively spanned by the differences $x-y$ with
$x,y\in X$, is called the \emph{augmentation ideal} of $A_X$ \cite{Ru24}*{Example~2}. It is a cofinite integral brace of dimension $|X|-1$. Following \cite{Ru24}*{Definition~2}, a subset $X$ of a brace $A$ is a \emph{cycle base}
if $X$ generates $(A,+)$ and $\lambda_a(X)=X$ for all $a\in A$. Such an $X$, with $r(x,y)=\big(\lambda_x(y),\lambda^{-1}_{\lambda_x(y)}(x)\big)$, is a solution. Conversely, Rump showed that every cofinite integral brace has a finite cycle base $X$ and is a socle reduction of $A_X$, i.e. a quotient of $A_X$ by an ideal contained in $\Soc(A_X)$ \cite{Ru24}*{Theorem~2 and Corollary, p.~288}. The question is therefore when such a quotient can be avoided. By \cite{Ru24}*{p.~287}, a brace $A$ with a cycle base $X$ is isomorphic to $A_X$ if and only if $(A,+)$ is free abelian with basis $X$. Hence a CIB is isomorphic to the brace of a structure group if and only if its additive group has a $\mathbb{Z}$-basis invariant under $\lambda(A)$. For instance, exactly four of the ten three-dimensional Bieberbach groups are structure groups of solutions \cite{Ru24}*{Example~12}. The Hantzsche--Wendt group is not one of them, but it is the multiplicative group of the augmentation ideal $I_X$ of a suitable $A_X$ \cite{Ru24}*{Example~14}.
\begin{problem}
Which cofinite integral braces are isomorphic to the brace of the structure group $G(X,r)$ of some solution $(X,r)$, or equivalently, admit a $\lambda$-invariant basis of the additive group? Which Bieberbach groups arise as structure groups of solutions, or as multiplicative groups of augmentation ideals $I_X$?
\end{problem}

\section*{Data availability}

The results dataset has been deposited at

\href{https://doi.org/10.17632/knvjwht7w4.1}{Low dimensional cofinite integral braces (Mendeley data)}

\bibliography{bibliography}

\end{document}